\documentclass[12pt]{amsart}
\usepackage{amsfonts, amssymb, amsgen, amsbsy, amstext, amsopn, amsmath,
latexsym, indentfirst,color,longtable,verbatim}
\newcommand{\G}{\mathbb{G}}

\newcommand{\N}{\mathbb{N}}

\newcommand{\R}{\mathbb{R}}

\newcommand{\cL}{\mathcal{L}}

\renewcommand{\div}{\mbox{\rm div}}

\newcommand{\dist}{\mbox{dist}}

\newcommand{\ep}{\varepsilon}

\newcommand{\ph}{\varphi}

\newcommand{\sm}{\setminus}

\newcommand{\ra}{\rightarrow}

\newcommand{\der}{\partial}

\newcommand{\lv}{\left|}
\newcommand{\rv}{\right|}
\newcommand{\cX}{\mathcal{X}}

\DeclareMathOperator{\espo}{e}
\newcommand{\eap}{\espo_{\textup{ap}}}

\newcommand{\medint}{\hbox{\vrule height3.5pt depth-2.8pt
width4pt}\mkern-14mu\int\nolimits}

\renewcommand{\medint}{\hbox{\vrule height3.5pt depth-2.8pt width4pt
}\mkern-14mu\int\nolimits}

\newtheorem{The}{Theorem}[section]
\newtheorem{Lem}[The]{Lemma}

\newtheorem{Def}[The]{Definition}

\newtheorem{Rem}[The]{Remark}

\newtheorem{Pro}[The]{Proposition}

\newtheorem{Cor}[The]{Corollary}

\begin{document}

\title
[Regularity estimates for convex functions in CC-spaces]
{\bf Regularity estimates for convex functions in Carnot-Carath\'eodory spaces}
\author{V. Magnani}
\address{Valentino Magnani, Dip.to di Matematica \\
Largo Pontecorvo 5 \\ 56127, Pisa, Italy}
\email{magnani@dm.unipi.it}
\author{M. Scienza}
\address{Matteo Scienza}
%\address{Matteo Scienza, Dip.to di Matematica \\ Largo Bruno Pontecorvo 5 \\ 56127, Pisa, Italy}
\email{scienza@mail.dm.unipi.it}

\begin{abstract}
We prove some first order regularity estimates for a class of convex functions in 
Carnot-Carath\'eodory spaces, generated by H\"ormander vector fields. 
Our approach relies on both the structure of metric balls induced by H\"ormander vector fields 
and local upper estimates for the corresponding subharmonic functions.
\end{abstract}

%\thanks{The first author acknowledges the support of the European Project ERC AdG *GeMeThNES*}
%\subjclass[2010]{Primary .... Secondary ....}
%\keywords{convexity, Carnot-Carath\'eodory spaces, Lipschitz estimates}
\date{\today}

\maketitle

\tableofcontents

\footnoterule{
\noindent
 The first author is supported by the European Project ERC AdG *GeMeThNES* 
}

\pagebreak
%
%
%
%
%%%%%%%%%%%%%%%%%%%%%%%%%%%%%%%%%%%%%%%%%%%%%%%%%%%%%%%%%%%%%%%%%%%%%%%%%%%%%%%%%%%%%%%%%%%%%%%%%%%%%%%%%%%%
%
%
%
%
\section{Introduction}
%
%
%
%
%%%%%%%%%%%%%%%%%%%%%%%%%%%%%%%%%%%%%%%%%%%%%%%%%%%%%%%%%%%%%%%%%%%%%%%%%%%%%%%%%%%%%%%%%%%%%%%%%%%%%%%%%%%%
%
%
% 
%

The present paper is devoted to the study of first order regularity properties of convex functions in Carnot-Carath\'eodory spaces. An important class of these spaces is that of Carnot groups, that can be seen as $\R^n$ equipped with both a group operation and a stratified Lie algebra of left invariant vector fields. By definition, this algebra is spanned by a choice of elements $X_1,\ldots, X_m$, along with their iterated commutators.
The latter condition is a special instance of the more general {\em H\"ormander condition} for any given set $\cX$ of vector fields. When we only assume that the set $\cX$ of vector fields satisfies this condition,
we obtain a {\em Carnot-Carath\'eodory space}. All linear combinations of elements of $\cX$ correspond to the so-called {\em horizontal vector fields}. These vector fields yield the well known 
{\em Carnot-Carath\'eodory distance}, hence they also generate the metric structure of the space, see Section~\ref{Sect:MainNotions} for precise definitions.
Convexity in this framework first appeared in Carnot groups \cite{DGN2003}, \cite{LMS2004}, \cite{JLMS2007}, then further extensions of this notion to general vector fields have been considered in \cite{Trud06}, \cite{BarDra2011}.

Convexity plays an important role in the regularity theory for second order elliptic non-divergence operators,
due to the Aleksandrov-Bakelman-Pucci estimate, \cite{CafCab95}. The project of extending this approach to subelliptic non-divergence operators was one of the main motivations for introducing convexity in Carnot groups, \cite{DGN2003}, \cite{DGN03}, \cite{LMS2004}, \cite{JLMS2007}. Other motivations come from the study of comparison principles for fully nonlinear degenerate subelliptic equations, \cite{BarDra2011}, \cite{BarMan2013}. 

After these works, the study of convexity in this non-Euclidean framework has known an increasing interest with several papers on topics like characterizations of convexity, Lipschitz continuity, subdifferentials, first and second order differentiability and monotonicity properties, \cite{BarDra2014}, \cite{Mag23BLSci}, \cite{BonLanTom2013}, 
\cite{DGNT2004}, \cite{GarTou2005}, \cite{GutMon1}, \cite{GutMon2}, \cite{JLMS2007}, \cite{Mag06}, \cite{Mag20Sci}, \cite{Rickly06}, \cite{Trud06}, \cite{TrudZhang13}, \cite{Wang05}, but this list could be certainly larger. 

A geometric approach to convex functions with respect to general vector fields was developed by Bardi and Dragoni in \cite{BarDra2011}, according to the following notion. If $\Omega\subset\R^n$ is open and $\cX=\{X_1,\ldots,X_m\}$ are $C^2$ smooth vector fields on $\R^n$, we say that $u: \Omega \rightarrow \R$ is {\em $\cX$-convex}, if $u\circ\gamma$ is convex, where $\gamma:I\ra\Omega$ satisfies $\dot \gamma=\sum_{i=1}^m \alpha_i\ X_i\circ\gamma$ on the open interval $I$ and $\alpha_i\in\R$ are arbitrary.
In analogy with the approach of \cite{JLMS2007}, {\em v-convexity} with respect to $\cX$ requires that 
\begin{equation}
\nabla_\cX^2\,u\geq0 \quad\mbox{in the viscosity sense.}
\end{equation}
It is interesting to notice that {\em in the class of upper semicontinuous functions, the notions of v-convexity and $\cX$-convexity do coincide}, where the vector fields of $\cX$ are assumed to be of class $C^2$. This characterization has been proved in \cite{BarDra2011}, along with Lipschitz continuity estimates of $\cX$-semiconvex functions in terms of the $L^\infty$-norm of the function, see Theorem~6.1 and Remark~6.2 of \cite{BarDra2011} for more details. In particular, here the vector fields are not required to satisfy the H\"ormander condition. 

On the other hand, the investigation of convex functions often requires stronger estimates on the Lipschitz constant. When $\cX$ generates a Carnot group structure, we have the strengthened estimate
%
%\sup_{w\in B_{x,r}} |u(w)|&\leq& C_0 \ \medint_{B_{x,2r}} |u(w)|\,dw 
%
\begin{equation}\label{gradest}
\mbox{\rm ess}\!\!\! \sup_{w\in B_{x,r}} |\nabla_Hu(w)|\leq \frac{C_0}{r} \ \medint_{B_{x,2r}} |u(w)|\,dw 
\end{equation}
for continuous weakly H-convex functions, \cite{DGN2003}, and upper semicontinuous v-convex functions, \cite{LMS2004}, \cite{JLMS2007}, where $x$ varies in $\G$, $r>0$ and $C_0>0$ is a suitable geometric constant depending on the metric structure of $\G$. Here $B_{x,r}$ denotes the metric ball with respect to the homogeneous distance fixed on the group and $\nabla_Hu$ is the {\em horizontal gradient} $(X_1u,\ldots,X_mu)$.
Let us point out that in Carnot groups the Lipschitz constant can be bounded by the $L^\infty$-norm of the horizontal gradient in a larger set, with controlled scaling, see for instance Lemma 6.1 of \cite{Mag14}. As a result, the estimate \eqref{gradest} immediately gives an integral upper estimate for the Lipschitz constant.

The same estimate plays an important role in the study of fine properties of convex functions in Carnot groups. This occurs for instance in relation to both the second order differentiability, see for instance \cite{Mag06}, and the distributional characterizations of convex functions, \cite{Mag23BLSci}.
The project of understanding these results in a broader context certainly requires first to study the validity of \eqref{gradest} for general H\"ormander vector fields.
This is precisely our main result, according to the next theorem.
\begin{The}\label{Thm:IE}
Let $\cX=\{X_1,\ldots,X_m\}$ be a set of H\"ormander vector fields, let $\Omega\subset\R^n$ be open and let $K\subset\Omega$ be compact.
Then there exist $C>0$ and $R>0$, depending on $K$, such that each $\cX$-convex function $u:\Omega\ra\R$, 
that is locally bounded from above, for every $x\in K$ satisfies the following estimates
\begin{eqnarray}\label{supestCC}
\sup_{B_{x,r}} |u|&\leq&  C \ \medint_{B_{x,2 r}}|u(w)|\,dw \\
|u(y)-u(z)|&\leq& C \ \frac{d(y,z)}{r} \  \medint_{B_{x,2 r}}|u(w)|\,dw\,, \label{gradestCC}
\end{eqnarray}
for every $0<r<R$ and every $y,z\in B_{x,r}$.
\end{The}
We first point out that \eqref{gradestCC} joined with Proposition~\ref{dirDer} immediately leads to \eqref{gradest}, hence
the previous theorem contains the known case of Carnot groups. However, in the proof Theorem~\ref{Thm:IE}, the absence of a group operation and of dilations compatible with the distance represents the source of new difficulties. In particular, this lack of homogeneity implies that the constant $C>0$ cannot be chosen independently of $K$, as it occurs for Carnot groups,
since in general Carnot-Carath\'eodory spaces the {\em doubling dimension} may change from point to point.

Let us present the main scheme to establish Theorem~\ref{Thm:IE}. 
The first point is to prove the Lipschitz continuity of a $\cX$-convex function $u:\Omega\ra\R$ that is only 
assumed to be locally bounded from above, then finding an upper estimate on its Lipschitz constant in terms of
$\|u\|_{L^\infty}$, see Theorem~\ref{LipReg}.
Let us point out that this theorem does not follow from \cite{BarDra2011}, since here the authors consider
$\cX$-semiconvex functions that are also assumed to be locally bounded and upper semicontinuous. 
In fact, the approach of \cite{BarDra2011} starts from the bound
on the horizontal gradient of the function in the viscosity sense, see Proposition 6.1 of \cite{BarDra2011},
then the upper semicontinuity assumption allows for translating this information into the wished Lipschitz estimate, see Lemma~6.1 of \cite{BarDra2011}. 

We are forced to use a completely different approach, since our $\cX$-convex function 
is only locally bounded from above, so in principle could not be even measurable. 
In fact, we use the stronger assumption that our vector fields satisfy the H\"ormander condition, hence we rely on the interesting result of \cite{Mo00}, that allows for covering the Carnot-Carath\'eodory ball by suitable compositions of flows of horizontal vector fields in a quantitative way, depending on the radius of the ball. This eventually leads to the proof of Theorem~\ref{LipReg}.

The previous step shows in particular that $u$ belongs to the anisotropic Sobolev space $W^{1,2}_{\cX,loc}(\Omega)$, see Section~\ref{Sect:MainNotions} for more information. The crucial point now is to show that for every $x\in\Omega$ the $\cX$-convex function $u$ is a {\em subharmonic} with respect to a suitable ``pointed sub-Laplacian'' $\cL_x=\sum_{j=1}^mY_j^2$, that is constructed around $x$.

This is the content of the following theorem.
\begin{The}\label{issub}
Let $\cX=\{X_1,\ldots,X_m\}$ be a set of H\"ormander vector fields, let $\Omega \subset \R^n$ be open, let $x_0\in\Omega$ and let $u:\Omega \rightarrow \R$ be a $\cX$-convex function that is locally bounded from above. 
There exist $\delta_0>0$ and a family of vector fields $\cX_1 =\{Y_1,\ldots,Y_m\}$, with $Y_i = \sum_{j=1}^m a_{ij} X_j$, and $a_{ij} \in \lbrace 0,1\rbrace$, both depending on $x_0$, such that $B_{x_0,\delta_0}\subset\Omega$ and
$u$ is a weak subsolution of the equation
\begin{equation}\label{subLapY}
\sum_{i = 1}^m Y_i^2 v = 0\quad\mbox{on}\quad B_{x_0,\delta_0}\,.
\end{equation}
\end{The}
%
%^^
%
Since the Lebesgue measure is locally doubling with respect to metric balls and the Poincar\'e inequality holds, the classical Moser iteration technique holds for weak subsolutions to the sub-Laplacian equation, hence getting the classical inequality
\begin{equation}\label{upeq1}
\sup_{B_{y,\frac{r}{2}}} u  \leq \kappa_x \; \medint_{B_{y,r}} |u(z)| dz
\end{equation}
for $0<r<\sigma_x$ and $y\in B_{x,\delta_x}$, where the positive constants $\kappa_x$,
$\sigma_x$ and $\delta_x>0$ depend on $x$, see Section~\ref{Sect:WeakSUI} for more information
and in particular Corollary~\ref{up}. 
The lower estimate of $u$ is reached using the {\em almost exponential} introduced in \eqref{Emap}, hence obtaining the following pointwise estimate 
\begin{equation}
2^{N_x}\, u(x) - (2^{N_x}-1)\sup_{B_{x,\bar N\delta}} u \leq \inf_{B_{x,b\delta}} u\,,
\end{equation}
where $N_x$ depends on $x$ and it satisfies the uniform inequality $1\leq N_x\leq \bar N$ on some compact set, see Lemma~\ref{lowbound}. This eventually leads us to the proof of \eqref{supestCC}.
The estimate \eqref{gradestCC} is obtained joining Theorem~\ref{LipReg} with Theorem~\ref{Thm:lambda}. 
In sum, the geometric part of our method arises from a quantitative representation of
the Carnot-Carath\'eodory ball by almost exponentials and it gives the lower estimates, then
the PDE approach leads to the upper estimates.

Our results have also an unexpected connection with the regularity of $k$-convex functions studied by Trudinger in the same framework of H\"ormander vector fields, see \cite{Trud06}. Here a smooth $k$-convex function has the pro\-perty that all $j$-th elementary symmetric functions of the horizontal Hessian $\nabla_\cX^2u$ are nonnegative for all $j=1,\ldots,k$ and $k\leq m$, where
$\cX=\{X_1,\ldots,X_m\}$. Then the nonsmooth $k$-convex functions are defined as $L^1_{loc}$-limits of smooth $k$-convex functions.

In the case $k=m$, it is not difficult to observe that \eqref{gradestCC} gives the local Lipschitz continuity of nonsmooth $m$-convex functions with respect to H\"ormander vector fields. In fact, these functions are $\cX$-convex. As a byproduct of this simple characterization, we can improve a family of estimates in \cite{Trud06}. According to these estimates, we have
\begin{equation}\label{eq:Holder}
 \sup_{\substack{x,y\in\Omega' \\ x\neq y} } \frac{|u(x)-u(y)|}{d(x,y)^\alpha}\le C\,\int_\Omega |u(x)|\,dx
\end{equation}
for any nonsmooth $k$-convex function $u:\Omega\to\R$, where $\Omega\subset\R^n$, $\Omega'$ is compactly contained in $\Omega$, $C$ is a geometric constant depending on $\Omega'$, 
\begin{equation}\label{alpha}
\alpha=\big(k(Q+m-2)-m(Q-1)\big)k^{-1}(m-1)^{-1}
\end{equation}
for every $k<m$ and $\alpha<1$ in the case $k=m$.
Our estimate \eqref{gradestCC} precisely shows that $\alpha$ can be chosen to be equal to one in the case $k=m$,
that fits with \eqref{alpha}.

We conclude by a short description of the paper. Section~\ref{Sect:MainNotions} recalls some elementary
facts on H\"ormander vector fields and Carnot-Carath\'eodory distances. 
In Section~\ref{Sect:AlmostExpCC}, we present the basic properties of the so-called almost exponential.
Section~\ref{Sect:BA} contains Theorem~\ref{LipReg} along with its proof. In Section~\ref{Sect:WeakSUI},
we use the local integral upper bounds for subharmonic functions to prove Theorem~\ref{issub}.
Section~\ref {Sect:IntEst} collects the preceding results in order to establish \eqref{supestCC} in Theorem~\ref{Thm:IE}.

\smallskip

{\bf Acknowledgements.} It is a pleasure to thank Daniele Morbildelli for some useful comments
on different notions of distances in Carnot-Carath\'eodory spaces.

%
%
%
%
%
%
%
%%%%%%%%%%%%%%%%%%%%%%%%%%%%%%%%%%%%%%%%%%%%%%%%%%%%%%%%%%%%%%%%%%%%%%%%%%%%%%%%%%%%%%%%%%%%%%%%%%%%%%%%%%%%
%
%
%
%
\section{Some basic notions and facts}\label{Sect:MainNotions}
%
%
%
%
%%%%%%%%%%%%%%%%%%%%%%%%%%%%%%%%%%%%%%%%%%%%%%%%%%%%%%%%%%%%%%%%%%%%%%%%%%%%%%%%%%%%%%%%%%%%%%%%%%%%%%%%%%%%%
%
%
% 
% 
Throughout the paper, we consider a family $\cX$ of smooth vector fields $X_1,\ldots,X_m$ on $\R^n$, 
which satisfy the {\em H\"ormander condition}: for every $x\in\R^n$ there exists a positive integer $r'$ such that

\begin{equation}\label{hormander}
\mbox{span}\lbrace X_{[S]}(x) : |S| \leq r' \rbrace = \R^n.
\end{equation}
For every multi-index $S = (s_1,\ldots,s_p)\in\{1,2,\ldots,m\}^p$, we have set $|S|= p$ and 
\begin{equation}\label{iteratedcommutators}
X_{[S]} =  \left[ X_{s_1},\left[ \ldots ,\left[ X_{s_{p-1}},X_{s_p}\right]\ldots\right]\right].
\end{equation}
\begin{Rem}\label{horm}\rm
As a consequence of the H\"ormander condition, for every bounded set $A\subset\R^n$ we have a positive
integer $r$ such that \eqref{hormander} is satisfied for $r'=r$ and all $x\in A$.
\end{Rem}
%
%
%                      FLOW OF VECTOR FIELDS 
% 
%
\begin{Def}[Flow of a vector field]\rm
Let $X$ be a smooth vector field of $\R^n$ and let $x\in\R^n$.
We consider the Cauchy problem
\[
\left \lbrace \begin{array}{l}
\dot \gamma(t) = X(\gamma(t))\\
\gamma(0) = x 
\end{array}\right.
\]
and denote its solution by $t\to \Phi^X(x,t)$. The mapping $\Phi^X$ defined on an open neighbourhood of
$\R^n\times\{0\}$ in $\R^{n+1}$ is the flow associated to $X$. The flow $\Phi^X$ will also define
the local diffeomorphism $\Phi^X_t(\cdot)=\Phi^X(\cdot,t)$ on bounded open sets for $t$
sufficiently small.
\end{Def}
%
%
%        FEFFERMAN PHONG   AND   FRANCHI LANCONELLI DISTANCE
%
%
\begin{Def}[CC-distances and metric balls]\label{d_FP}\rm
For every $x,y \in \R^n$ we define the following distance
\begin{equation}
d(x,y) =  \inf \lbrace t>0 :\ \mbox{there exists}\ \gamma \in \Gamma_{x,y}(t)\rbrace\,,
\end{equation}
where $\Gamma_{x,y}(t)$ denotes the family of all absolutely continuous curves 
$\gamma:[0,t]\ra \R^n$ with $\gamma(0)=x$, $\gamma(t)=y$ and such that for a.e. $s\in[0,t]$ we have
\[
\dot\gamma(s) = \sum_{j=1}^m a_j(s) X_j(\gamma(s)) \quad \mbox{and} \quad 
\max_{1\leq j\leq m}|a_j(s)|\leq 1 \,.
\]
This distance along with its properties can be found in \cite{NSW}. If the previous
condition is modified replacing $\max_{1\leq j\leq m}|a_j(s)|$ with $(\sum_{1\leq j\leq m}a_j(s)^2)^{1/2}$, 
then in the context of PDEs this distance first appeared in a work by Fefferman and Phong, \cite{FefPho81}.
Metric balls are defined using the following notation
\[ 
B_{x,r}=\{z\in\R^n:d(z,x)<r\},\quad D_{x,r}=\{z\in\R^n:d(z,x)\leq r\}
\]
for any $r>0$ and $x\in\R^n$.
We say that $d$ is the Carnot-Carath\'eodory distance, in short {\em CC-distance}, 
with respect to $\cX$. Another analogous distance that will be important for the sequel
is the following one. Let $\Gamma_{x,y}^c(t)$ be the family of all absolutely continuous curves 
$\gamma:[0,t]\ra \R^n$ with $\gamma(0)=x$, $\gamma(t)=y$, such that for a.e. $s\in[0,t]$ we have
\[
\dot\gamma(s) = \sum_{j=1}^m a_j(s) X_j(\gamma(s)) \quad \mbox{and} \quad 
(a_1,\ldots,a_m)\in\{\pm e_1,\ldots,\pm e_m \}\,,
\]
where the curve $(a_1,\ldots,a_m)$ is piecewise constant on $[0,t]$ and $(e_1,\ldots,e_m)$
is the canonical basis of $\R^m$. Thus, we define the distance
\begin{equation}
\rho(x,y) =  \inf \lbrace t>0 :\ \mbox{there exists}\ \gamma \in \Gamma^c_{x,y}(t)\rbrace\,.
\end{equation}
This distance in the framework of PDEs has been first introduced by Franchi and Lanconelli, 
\cite{Franchi83}, \cite{Lanconelli83}, \cite{FL83}.
\end{Def}
\begin{Rem}\label{LipdexpX}\rm
Let us consider $X \in \cX$ and $t,\tau \in \R$, by definition of $d$ and $\rho$, we have
\[
\max\{d(\Phi_t^{X}(x), \Phi_\tau^X(x)),\rho(\Phi_t^{X}(x), \Phi_\tau^X(x))\} \leq |t - \tau|
\]
for any $x\in\R^n$, whenever the flows are defined for times $t$ and $\tau$.
\end{Rem}
%
%
%                                NSW DISTANCE
%
%
\begin{Rem}\label{equivd}\rm
Let $\cX$ be the family of smooth H\"ormander vector fields $X_1,\ldots,X_m$
introduced in Section~\ref{Sect:MainNotions}. Then by a rescaling argument, one can easily check
that there holds
\begin{equation}
d(x,y) = \inf \left \{\delta>0 : \mbox{there exists}\ \gamma\in \Gamma_{x,y}^\delta \right\}\,,
\end{equation}
where $\Gamma_{x,y}^\delta(\cX)$ is the family of absolutely continuous curves 
$\gamma:[0,1] \rightarrow \R^n$ such that $\gamma(0) = x$, $\gamma(1) = y$ and for a.e. $t\in[0,1]$ we have
\[
\dot\gamma(t) = \sum_{j=1}^m a_j(t)X_j(\gamma(t)) \quad \mbox{and}\quad \max_{1\leq j\leq m}|a_j(t)| < \delta\,,
\]
where $d$ is introduced in Definition~\ref{d_FP}.
\end{Rem}
\begin{Lem}\label{XbarX}
Let $d$ and $d_1$ two CC-distances associated to the families of smooth 
H\"ormander vector fields $\cX =\{X_1,\ldots,X_m\}$ and $\cX_1 = \{Y_1,\ldots,Y_m\}$, respectively. 
Let $\{i_1,j_1,\ldots,j_{m-1}\}=\{1,2,\ldots,m\}$ and assume that $Y_j = X_j$ for all $j\neq i_1$ and 
$Y_{i_1} = X_{i_1} + X_{j_1}$. Then we have $4^{-1}d\leq d_1\leq 4d$.
\end{Lem}
\begin{proof}
We can use for $d$ and $d_1$ the equivalent definition stated in Remark~\ref{equivd}.
Taking this into account, we fix a compact set $K\subset\R^n$ and choose any $x_1,x_2\in K$, 
setting $d(x_1,x_2) =\delta/2$, for some $\delta>0$. Then there exists an 
absolutely continuous curve $\gamma : [0,1] \ra \R^n$ belonging to $\Gamma^\delta_{x,y}(\cX)$.
Clearly, we observe that
\[
\dot \gamma =a_{i_1} Y_{i_1}(\gamma)+ (a_{j_1}-a_{i_1}) Y_{j_1}(\gamma) + 
\sum_{s=2}^{m-1} a_{j_s}\,Y_{j_s}(\gamma),
\]
hence $\gamma\in\Gamma_{x,y}^{2\delta}(\cX_1)$, then $d_1(x,y)\leq2\delta=4\,d(x,y)$.
In analogous way we get $ d(x_1,x_2) \leq 4\, d_1(x_1,x_2)$, concluding the proof.
\end{proof}
Next, we introduce the anisotropic Sobolev space $W_\cX^{1,p}$ with respect to the family $\cX$. 
Throughout, for every open set $\Omega \subset \R^n$ we denote by $C^\infty_c(\Omega)$, the class of 
smooth functions with compact support.
\begin{Def}\rm
Given an open  set $\Omega \subset \R^n$, we define the {\em $\cX$-Sobolev space}
$W_\cX^{1,p}(\Omega)$, with $ 1\leq p \leq \infty$, as follows
\[
W_\cX^{1,p}(\Omega)  = \left \lbrace  f\in L^p(\Omega),\ X_jf \in L^p(\Omega),\ j=1,\ldots,m\right \rbrace,
\]
where $X_ju$ is the {\em distributional derivative} of $u \in L^1_{loc}(\Omega)$, namely
\[
\langle X_i u , \phi \rangle = \int_\Omega u\ X_i^* \phi\ dx, \quad \phi \in C^{\infty}_0(\Omega),
\]
and $X_i^*$ is the formal adjoint of $X_i$, namely, $X_i^*=-X_i-\div X_i$.
\end{Def}
The linear space $W^{1,p}_\cX(\Omega)$ is turned into a Banach space by the norm 
\[
 \Vert f \Vert_{W_\cX^{1,p}(\Omega)} := \Vert f \Vert_{L^p(\Omega)} + \sum_{j=1}^m \Vert X_i f \Vert_{L^p(\Omega)}\,.
\]
A function $u \in W_\cX^{1,2}(\Omega)$ is an {\em $\cL$-weak subsolution} of
\begin{equation}\label{subLap}
\cL u = \sum_{i = 1}^m X_i^2 u = 0,
\end{equation}
if for every nonnegative $\eta \in W_{\cX,0}^{1,2}(\Omega)$, we have
$\displaystyle \sum_{i=1}^m \int_{\Omega} X_i u X^*_i \eta dx \geq 0$.
\begin{Lem}\label{flows}
Let $\Omega'$ be an open set compactly contained in $\Omega$ and let $X \in \cX$.
There exists $T >0$ such that  the map  $\Phi^X$ is well defined on $\Omega'\times (-2T,2T)$
and for every $t \in (-2T,2T)$, the mapping $\Phi^X(\cdot,t)  : \Omega' \ra \R^n$ 
is bi-Lipschitz onto its image with inverse $\Phi^X(\cdot,-t)$. The Jacobian 
$J_X$ of $\Phi^X$ satisfies 
\[
J_X(x,t) = 1 + \tilde J_X(x,t)\quad\mbox{and}\quad |\tilde J_X(x,t)| \leq C|t|
\]
for all $x\in\Omega'$ and $|t|<2T$, where $C>0$ is independent of $x$ and $t$.
\end{Lem}
The proof of this lemma can be achieved by standard ODEs methods,
see also \cite{FSSC97} for the general case of a Lipschitz vector field.
\begin{The}\label{existenceDer}
Every Lipschitz function on an open set $\Omega\subset \R^n$ belongs to $W_\cX^{1,\infty}(\Omega)$.
\end{The}
The proof of Theorem~\ref{existenceDer} can be found from either Proposition~2.9 of \cite{FSSC97} or
Theorem~1.3 of \cite{GN98}. From either these papers or the arguments of Theorem 11.7 of \cite{HajKos2000},
it is also not difficult to deduce the following proposition.
\begin{Pro}\label{dirDer}
Let $u : \Omega \ra \R$ be a Lipschitz function. Let $X$ be a vector field of $\cX$ and  fix $ x\in \Omega$. 
Let $\Phi^X_t(x)$ be the flow of $X$ starting at $x$. Then the directional derivative 
$\frac{d}{dt} u(\Phi^X_t(x))|_{t = 0}$ exists almost everywhere and it coincides
with the distributional derivative $Xu$.
\end{Pro}
%
%
%
%
%%%%%%%%%%%%%%%%%%%%%%%%%%%%%%%%%%%%%%%%%%%%%%%%%%%%%%%%%%%%%%%%%%%%%%%%%%%%%%%%%%%%%%%%%%%%%%%%%%%%%%%%%%%%
%
%
%
%
\section{Almost exponentials and CC-distances}\label{Sect:AlmostExpCC}
%
%
%
%
%%%%%%%%%%%%%%%%%%%%%%%%%%%%%%%%%%%%%%%%%%%%%%%%%%%%%%%%%%%%%%%%%%%%%%%%%%%%%%%%%%%%%%%%%%%%%%%%%%%%%%%%%%%%%
%
%
% 
% 
In this section, we introduce a kind of ``discrete exponential mappings'' for vector fields and recall their properties, following notations and results of \cite{Mo00}. We define 
\begin{eqnarray}\nonumber
 \cX^{(1)} &=& \lbrace X_1,\ldots,X_m\rbrace,\\ \nonumber
 \cX^{(2)} &=& \lbrace X_{[i_1,i_2]},\  1 \leq i_1 < i_2 \leq m \rbrace
\end{eqnarray}
and so on, in such a manner that elements of $\cX^{(k)}$ are the commutators of length $k$. 
We denote by $Y_1,\ldots,Y_q$ an enumeration of all the elements of
$\cX^{(1)},\ldots, \cX^{(r)}$, where $r$ is an integer large enough to ensure that $Y_1,\ldots,Y_q$ span $\R^n$ at each point of a fixed bounded open set $\Omega \subset \R^n$, see Remark~\ref{horm}. 
We call $r$ the {\em local spanning step} and $q$ the {\em local spanning number} of $\cX$, 
to underly that they depend on $\Omega$. It may be worth to stress that the Lie algebra spanned by $\cX$ 
at some point need not be nilpotent, although the local spanning step is finite.

If $Y_i$ is an element of $\cX^{(j)}$, we say $Y_i$ has formal degree $d_i := d(Y_i) = j$.
Let $I = (i_1,\ldots,i_n)\in\{1,2,\ldots,q\}^n$ be a multi-index and define from \cite{NSW}
the functions
$$
\lambda_I(x) = \mbox{det}\left[ Y_{i_1}(x),\ldots, Y_{i_n}(x) \right] \quad\mbox{and}\quad
\Vert h \Vert_I = \max_{j= 1,\ldots,n} |h_j|^{
\frac{1}{d(Y_{i_j}) }}.
$$
As a consequence of the choice of $(Y_1,\ldots,Y_q)$, we have that for every $x\in\Omega$ there exists 
$I\in \{1,2,\ldots,q\}^n$ with $\lambda_I(x)\neq0$.
We denote by $d(I)$ the integer $d_{i_1} + \ldots + d_{i_n}$, where $d_{i_k} = d(Y_{i_k})$.
%
%
%             LOCAL EXPONENTIAL AND LOCAL PRODUCT
%
%
\begin{Def}\rm
Let $X,S\in\cX$ and consider the mappings $\Phi^X_t$ and $\Phi^S_t$, that coincide with 
$\Phi^{tX}_1$ and $\Phi^{tS}_1$, respectively.
Thus, for $t$ sufficiently small, we can define the {\em local exponentials} 
$\exp (t X):=\Phi^{tX}_1$ and $\exp(t S):=\Phi^{tS}_1$, along with the {\em local product}
\[
\exp(t X)\exp(tS)=\Phi^{tX}_1\circ\Phi^{tS}_1\,.
\]
\end{Def}
Let $S_1,\ldots,S_l$ be vector fields belonging to the family $\cX$. Therefore,
for every $a\in \R$ sufficiently small, we can define 
\begin{equation*}
\begin{aligned}
&C_1(a,S_1) = \mbox{exp}(a S_1),\\ 
&C_2(a,S_1,S_2) = \mbox{exp}(-a S_2)\mbox{exp}(-a S_1)\mbox{exp}(a S_2)\mbox{exp}(a S_1),\\ \nonumber
&C_l(a,S_1,\ldots,S_l) =  C_{l-1}(a;S_2,\ldots,S_l)^{-1}\mbox{exp}(-a S_1)C_{l-1}(a;S_2,\ldots,S_l)\mbox{exp}(a S_1).
\end{aligned}
\end{equation*}
By (14) of \cite{Mo00}, for $\sigma \in \R$ sufficiently small we define the {\em approximate exponential}
\begin{equation} \label{almostexp}
\eap^{\sigma S_{[(1,\ldots,l)]}} = 
\left\lbrace
\begin{array}{lcr}
C_l(\sigma^{\frac{1}{l}},S_1,\ldots,S_l), & & \sigma>0, \\
C_l(|\sigma|^{\frac{1}{l}},S_1,\ldots,S_l)^{-1}, & & \sigma<0.
\end{array}
\right.
\end{equation}
Following (16) of \cite{Mo00}, given a multi-index $I = (i_1,\ldots,i_n)$, 
$1\leq i_j \leq q$ and $h \in \R^n$ small enough, then we introduce the {\em almost exponential}
\begin{equation}\label{Emap}
E_{I,x}(h) = \eap^{h_1 Y_{i_1}}\cdots \eap^{h_n Y_{i_n}}(x).
\end{equation}
The next theorem, that is contained in Theorem~3.1 of \cite{Mo00}, shows that the 
almost exponentials give a good representation of the Carnot-Carath\'eodory balls.
%
%
%                MORBIDELLI'S THEOREM
%
%
\begin{The}\label{bounds}
If $\Omega\subset\R^n$ is an open bounded set with local spanning number $q$ and
$K \subset \Omega$ is a compact set, then there exist $\delta_0>0$ and positive 
numbers $a$ and $b$, $b<a<1$, so that, given any $I\in\{1,\ldots,q\}^n$ such that
\begin{equation}\label{lambdacond}
|\lambda_I(x)|\delta^{d(I)} \geq \frac{1}{2} \max_{J\in\{1,\ldots,q\}^n}|\lambda_J(x)| \delta^{d(J)},
\end{equation}
for  $x \in K$ and $0<\delta < \delta_0$, it follows that 
$
B_{x,b\delta} \subset E_{I,x}(\{h\in\R^n: \Vert h \Vert_I < a\delta\}) \subset B_{x,\delta}.
$
\end{The}
Following the terminology of \cite{NSW}, we introduce the following definition.
\begin{Def}\label{equivdists}\rm 
We say that two distances $\rho_1$ and $\rho_2$ in $\R^n$ are {\em equivalent}, if for every compact
set $K\subset\R^n$, there exist $c_K\geq1$, depending on $K$, such that 
\[ 
c_K^{-1} \rho_1(x,y)\leq\rho_2(x,y)\leq c_K\rho_1(x,y)\quad\mbox{ for all }\quad x,y\in K.
\]
\end{Def}
\begin{Rem}\rm
We have stated Theorem~\ref{bounds} using only metric balls with respect to the distance $d$.
In fact, in \cite{Mo00} the same symbol denotes the same distance, with a different definition,
see Remark~\ref{equivd}. Up to a change of the constant $b>0$ in Theorem~3.1 of \cite{Mo00}, we 
can replace the distance denoted by ''$\rho$'' in \cite{Mo00} with $d$. 
In fact, these two distances are equivalent, due to Theorem~4 of \cite{NSW}, joined with
our Remark~\ref{equivd}. 
\end{Rem}
The following proposition has been pointed out to us by D. Morbidelli. It is a consequence 
of the seminal paper by A. Nagel, E. M. Stein and S. Wainger \cite{NSW}, and it can be
also found as a consequence of Theorem~3.1 of \cite{Mo00}.
\begin{Pro}\label{equivrhod}
The distances $d$ and $\rho$ introduced in Definition~\ref{d_FP} are equivalent.
\end{Pro}
\begin{Rem}\rm
Notice that the inequality $d\leq \rho$ is trivial.
As a consequence of the previous proposition, $\cX$-convex functions that are locally bounded
are also locally Lipschitz continuous with respect to $d$ and any other equivalent distance,
according to the notion of equivalence given in Definition~\ref{equivdists} 
\end{Rem}
We fix a multi-index $I = (i_1,\ldots,i_n)$ and for each $Y_{i_k}$ we have a multi index 
\[
J_{i_k} = (j^{i_k}_1,j^{i_k}_2,\ldots,j^{i_k}_{d_{i_k}})\quad\mbox{such that }\quad Y_{i_k} = X_{[J_{i_k}]},
\]
where $d_{i_k}$ is the formal degree of $Y_{i_k}$.  
We notice that $1\leq j^{i_k}_s\leq m$ for all $1\leq s\leq d_{i_k}$ and $d_{i_k}\leq r$,
where $r$ is the local spanning step of $\cX$.
By definition of $\eap$ we get
\begin{equation}\label{eapk}
\eap^{h Y_{i_k}} = 
\left \lbrace
\begin{array}{lr}
\prod_{s=1}^{N_{i_k}}\mbox{exp}(\sigma_s h^\frac{1}{d_{i_k}} X^{i_k}_s) & h\geq 0,\\
\prod_{s=1}^{N_{i_k}}\mbox{exp}( -\sigma_{N_{i_k}+1-s} |h|^\frac{1}{d_{i_k}} X^{i_k}_{N_{i_k}+1-s})& h< 0.
\end{array} \right.
\end{equation}
where $\sigma_s \in \lbrace -1,1\rbrace$, $N_{i_k}$ is the length of 
$\eap^{hY_{i_k}}$ and $X^{i_k}_1,\ X^{i_k}_2,\ldots, X^{i_k}_{N_{i_k}}$ is a suitable possibly iterated
choice among the vectors $X_{j^{i_k}_1},X_{j^{i_k}_2},\ldots, X_{j^{i_k}_{d_{i_k}}}$.
A simple calculation gives  $N_{i_k} = 2^{d_{i_k}} - 2 + 2^{d_{i_k}-1}$. 
We define $N(I) = \sum_{k=1}^n 2 N_{i_k}$ along with the mapping $G_{I,x} : \R^N \rightarrow \R^n$,
that is
%
%
%                   DEFINITION OF G_{I,x}
%
%
\begin{equation}\label{GI}
G_{I,x}(w)= \prod_{k=1}^n \left \lbrace 
\prod_{s=1}^{N_{i_k}} \mbox{exp}(w_{k,s,2}X^{i_k}_{N_{i_k}+1-s})
\prod_{s=1}^{N_{i_k}} \mbox{exp}(w_{k,s,1}X^{i_k}_s)
 \right \rbrace(x).
\end{equation}
In the definition of $G_{I,x}$, we use the product to indicate the composition of 
flows according to the order that starts from the right.
The variable $w$ denotes the vector
\[
(w_{1,1,1},w_{1,2,1},\ldots,w_{1,N_{i_1},1},w_{1,1,2},w_{1,2,2},\ldots,w_{1,N_{i_1},2},
\ldots,w_{n,1,2},\ldots,w_{n,N_{i_n},2})
\]
belonging to $\R^{N(I)}$.
The integer $N(I)$ is locally uniformly bounded from above, since every multi-index
$I=(i_1,\ldots,i_n)$ of Theorem~\ref{bounds} depends on $x$ and satisfies $N_{i_k} \leq 2^{r}- 2 + 2^{r-1}$,
where $r$ is the local spanning step of $\cX$, depending on the fixed bounded open set $\Omega$.
Therefore we have a local upper bound $\bar N$ defined as follows
%
%
% definition of \bar N
%
%
\begin{equation}\label{barN}
\bar N = 2n(2^{r+1}-2 + 2^{r-1})
\end{equation}
and clearly $N(I)\leq \bar N$, where $\bar N$ is independent of $I$.
\begin{Def}\label{normN}\rm
For every $N\in\N\sm\{0\}$, we set $\displaystyle \Vert w \Vert_N = \max_{k=1,\ldots,N} |w_k|$,
for every $w \in \R^N$. The corresponding open ball is defined as follows
\[
S_{N,\delta}=\{w\in\R^N:\|w\|_N<\delta\}\,.
\]
\end{Def}
From standard theorems on ODEs, one can establish the following fact.
\begin{Pro}\label{ExN}
If $K\subset\Omega$ is a compact set and $N\in\N$ is positive, then there exists 
$\delta_1>0$ only depending on $K$, $\Omega$ and ${\mathcal X}$ such that for every 
$0<\delta\leq\delta_1$ and every $x\in K$ we have
$B_{x,N\delta_1}\subset\Omega$ and for every integers $1\leq j_1,\ldots,j_N\leq m$, the composition
\[ 
\Big(\exp(w_N X_{j_N})\cdots\exp(w_2 X_{j_2}) \exp(w_1 X_{j_1})\Big)(x) 
\]
is well defined and contained in $B_{x,N\delta}$ for all $w\in S_{N,\delta}$.
\end{Pro}
The previous proposition immediately leads us to the following consequence.
%
%
%            FIRST INCLUSION
%
%
\begin{Cor}\label{subset}
Let $\Omega$ be an open bounded set with local spanning number $q$ and local spanning step $r$. 
If $K\subset\Omega$ is a compact set, then there exist $\delta_1>0$ 
such that for every $x\in K$, every $0<\delta\leq\delta_1$ and every multi-index $I\in\{1,2,\ldots,q\}^n$, 
the mapping $G_{I,x}$ introduced in \eqref{GI} is well defined on $S_{N(I),\delta}$ and
\[
G_{I,x}(S_{N(I),\delta})\subset B_{x,\bar N \delta}\subset D_{x,\bar N\delta_1}\subset\Omega\,,
\]
where $\bar N$ is defined in \eqref{barN}.
\end{Cor}
For any of the above multi-indexes $I= (i_1,\ldots,i_n)$,
we introduce the function $F_{I,x} : \R^n \rightarrow \R^{N(I)}$ as follows
\begin{eqnarray*}
F_{I,x}(h_1,\ldots,h_n) = (\sigma_{1,1} \delta_1(h_1) h_{1}^{\frac{1}{d_{i_1}}},\ldots,\sigma_{1,N_{i_1}} \delta_1(h_1) h_{1}^{\frac{1}{d_{i_1}}}, -\sigma_{1,N_{i_1}} \delta_2(h_1) |h_{1}|^{\frac{1}{d_{i_1}}}, \\ 
\ldots,-\sigma_{1,1} \delta_2(h_1)|h_{1}|^{\frac{1}{d_{i_1}}},\ldots,
 \sigma_{n,1}\delta_1(h_n) h_n^{\frac{1}{d_{i_n}}} \ldots,\sigma_{n,N_{i_n}}\delta_1(h_n)h_{n}^{\frac{1}{d_{i_n}}},\ldots)
\end{eqnarray*}
where $\sigma_{k,j} \in \lbrace -1,1 \rbrace$, $k=1,\ldots,n$ and $j=1,\ldots,N_{i_k}$. 
More precisely, we have
%
%
%                DEFINITION OF F_{I,x}
%
%
\begin{equation}\label{FI}
F_{I,x}(h)=\sum_{k=1}^n\Big\lbrace \sum_{s=1}^{N_{i_k}} \sigma_{k,s}\,\delta_1(h_k)\,
h_k^{1/d_{i_k}}\,e_{k,s,1}-\sum_{s=1}^{N_{i_k}} \sigma_{k,N_{i_k}+1-s}\;\delta_2(h_k)
\,|h_k|^{1/d_{i_k}}\,e_{k,s,2}
\Big\rbrace\,,
\end{equation}
where we have introduced the canonical basis 
\[
\big\lbrace e_{k,s,i}: 1\leq k\leq n,\ 1\leq s\leq N_{i_k},\ i=1,2\big\rbrace
\]
of $\R^{N(I)}$ and the functions  
\[
\delta_1(x)= \left \lbrace \begin{array}{lr} 1& x\geq0 \\ 0 & x <0 \end{array}\right.
\quad\mbox{and}\quad
\delta_2(x)= \left \lbrace \begin{array}{lr} 0& x\geq0 \\1  & x <0 \end{array}\right.\,.
\]
%
%
%               E_{I,x}=G_{I,x}\circ F_{I,x}
%
%
\begin{Rem}\label{EGF}\rm 
From the definitions of $G_{I,x}$ and $F_{I,x}$, it is straightforward to observe that 
$ E_{I,x} = G_{I,x} \circ F_{I,x}$ on a sufficiently small neighbourhood of the origin in $\R^n$.
\end{Rem}
%
%
%                   DOUBLE INCLUSION
%
%
\begin{The}\label{supset}
If $\Omega\subset\R^n$ is an open bounded set with local spanning number $q$ and $K \subset \Omega$ is compact, then there exist $\delta_0>0$ and positive numbers $a$ and $b$, $b<a<1$, so that for any 
$x\in K$ and $0<\delta<\delta_0$ and any $I\in\{1,\ldots,q\}^n$ with
\begin{equation}\label{lamI}
|\lambda_I(x)|\delta^{d(I)} \geq \frac{1}{2} \max_{J\in\{1,\ldots,q\}^n}|\lambda_J(x)| \delta^{d(J)}\,,
\end{equation}
we have $B_{x,b\delta} \subset G_{I,x}(\{w\in\R^{N(I)}: \Vert w \Vert_{N(I)} < a\delta\}) 
\subset B_{x,\bar N\delta}\subset D_{x,\bar N \delta_0} \subset\Omega$.
\end{The}
\begin{proof}
From Theorem~\ref{bounds}, we get the existence of $\delta_0,a,b>0$, with $b<a<1$ such that 
for every $x\in K$, $0<\delta<\delta_0$ and $I\in\{1,\ldots,q\}^n$ satisfying \eqref{lamI},
we have the inclusion 
\[
B_{x,b\delta} \subset E_{I,x}(\{h\in\R^n: \Vert h\Vert_I<a\delta\}).
\]
This proves the validity of this inclusion, since for every $x\in K$ and $0<\delta<\delta_0$
the existence of $I$ satisfying \eqref{lamI} is trivial. From formula \eqref{FI}, we have 
\begin{equation}\label{VF_I}
\Vert F_{I,x}(h) \Vert_{N(I)} = \Vert h \Vert_I\quad\mbox{for all}\quad h\in\R^n.
\end{equation}
Remark~\ref{EGF} implies that $E_{I,x}(h)=G_{I,x}\circ F_{I,x}(h)$ for all $h\in\R^n$, possibly small,
such that $G_{I,x}$, introduced in \eqref{GI}, is well defined on $F_{I,x}(h)$. In view of Corollary~\ref{subset}, it is not 
restrictive to choose $\delta_0>0$ possibly smaller, such that $G_{I,x}$ is well defined on 
\begin{equation}\label{GSI}
S_{N(I),\delta_0}\quad \mbox{and}\quad
G_{I,x}(S_{N(I),\delta})\subset B_{x,\bar N\delta}\subset D_{x,\bar N\delta_0}\subset\Omega.
\end{equation}
Taking into account \eqref{VF_I}, we have
$F_{I,x}(\lbrace h\in\R^n: \Vert h \Vert_I < a\delta \rbrace ) \subset S_{N(I),\delta}$, that leads us to
the following inclusions
\begin{equation}
B_{x,b\delta}\subset E_{I,x}(\lbrace h\in\R^n: \Vert h \Vert_I < a\delta \rbrace ) \subset G_{I,x}\big(S_{N(I),\delta}\big)
\subset B_{x,\bar N\delta}
\end{equation}
concluding the proof.
\end{proof}
According to \cite{NSW}, for $x\in\R^n$, we set 
$$
\Lambda (x,\delta) = \sum_{I\in\{1,2,\ldots,q\}^n} \lv \lambda_I (x)\rv\delta^{d(I)}\,.
$$
From Theorem~1 of \cite{NSW}, we get the following important fact.
\begin{The}\label{volume1}
For every $K \subset \R^n$ compact, there exist $\delta_0>0$ and positive constants 
$C_1$ and $C_2$, depending on $K$, so that for all $x \in K$ and 
every $0<\delta<\delta_0$ we have
$$
C_1 \leq \frac{\lv B_{x,\delta}\rv}{\Lambda(x,\delta)} \leq C_2.
$$
\end{The}
The point of this theorem is that it gives the doubling property of metric balls, as pointed out in \cite{NSW}.
In fact, $\Lambda$ is a polynomial with respect to $\delta$, that only depends on the enumeration of
vector fields $Y_1,\ldots, Y_q$ on some fixed open bounded set $\Omega$. 
Thus, we have the following corollary.
%
%
%               DOUBLING PROPERTY OF METRIC BALLS
% 
%
\begin{Cor}\label{doubling}
For every compact set $K \subset \R^n$ there exist positive constants $C$ and $r_0$, depending on $K$,
such that for every $x \in K$ and every $0<r<r_0$, we have 
\[
|B_{x,2r}| \leq C\ |B_{x,r}|.
\]
\end{Cor}
%
%
%
%
%%%%%%%%%%%%%%%%%%%%%%%%%%%%%%%%%%%%%%%%%%%%%%%%%%%%%%%%%%%%%%%%%%%%%%%%%%%%%%%%%%%%%%%%%%%%%%%%%%%%%%%%%%%%
%
%
%
%
\section{Boundedness from above implies Lipschitz continuity}\label{Sect:BA}
%
%
%
%
%%%%%%%%%%%%%%%%%%%%%%%%%%%%%%%%%%%%%%%%%%%%%%%%%%%%%%%%%%%%%%%%%%%%%%%%%%%%%%%%%%%%%%%%%%%%%%%%%%%%%%%%%%%%%
%
%
%
%
%
%
This section is devoted to the proof of the local Lipschitz continuity of $\cX$-convex functions
that are locally bounded from above. Precisely, we will prove Theorem~\ref{LipReg}.

%
%
%
%      LOCAL BOUNDEDNESS FROM ABOVE IMPLIES LOCAL BOUNDEDNESS FROM BELOW
% 
%
\begin{Lem} \label{lowbound}
Let $u : \Omega \rightarrow \R$ be a $\cX$-convex function on an open set $\Omega\subset\R^n$
and let $K$ be a compact set. Then there exist $\delta_0>0$, $0<b<1$ and an integer $\bar N$ only depending on $K$ and $\cX$ such that for every $x\in K$, there exists an integer $1\leq N_x\leq\bar N$ such that 
for every $0<\delta<\delta_0$ we have $D_{x,\bar N\delta_0}\subset\Omega$ and
\begin{equation}\label{NxB}
2^{N_x}\, u(x) - (2^{N_x}-1)\sup_{B_{x,\bar N\delta}} u \leq \inf_{B_{x,b\delta}} u\,.
\end{equation}
\end{Lem}
\begin{proof}
Let $\Omega'$ be an open bounded set containing $K$ such that $\overline{\Omega'}\subset\Omega$,
let $r$ be the local spanning step and $q$ be the local spanning number with local spanning frame $Y_1,\ldots,Y_q$ on $\Omega'$. We apply Theorem~\ref{supset} to both $K$ and $\Omega'$,
getting an integer $\bar N$ and positive number $\delta_0>0$, $0<b<a<1$, depending on $K$,
$\Omega'$ and $\cX$, having the properties stated in this theorem.
Thus, we choose any $x\in K$ and $0<\delta<\delta_0$, so that we can find a multi-index
$I\in\{1,\ldots,q\}^n$ such that \eqref{lamI} holds. Theorem~\ref{supset} implies that
\[
B_{x,b\delta}\subset G_{I,x}(S_{N(I),a\delta})\subset B_{x,\bar N\delta}
\subset D_{x,\bar N\delta_0}\subset\Omega'
\]
where $G_{I,x}$ is defined in \eqref{GI}. In particular, the closure 
$\overline{ S}_{N(I),x}$ satisfies 
\[
G_{I,x}(\overline{ S}_{N(I),a\delta})\subset\Omega'.
\]
Let us consider the scalar function $\varphi(w)=u\circ G_{I,x}(w)$, that is well defined for all 
$w\in\overline{S}_{N(I),a\delta}$. By definition of $\cX$-convexity, we have
$$
\mu_1=2\varphi(0) - \sup_{B_{x,\bar N\delta}}u \leq 2\varphi(0) - \varphi(-w_1,0,\ldots,0) 
\leq \varphi(w_1,0,\ldots,0)\,,
$$
whenever $|w_1|\leq a\delta$. Notice that $\mu_1=2\, u(x)-\sup_{B_{x,\bar N\delta}}u$. Of course,
in the case $\sup_{B_{x,\bar N\delta}}u=+\infty$, then the inequalities \eqref{muN1} become trivial.
For each $w_1\in[-a\delta,a\delta]$, the function 
\[
[-a\delta,a\delta]\ni s \mapsto \ph(w_1,s,0,\ldots,0),
\]
is convex with respect to $s$, hence arguing as before we get 
\[
\mu_2 = 2\mu_1 - \sup_{B_{x,\bar N\delta}}u \leq 
\ph(w_1,s,0,\ldots,0)).
\]
whenever $|s|\leq a\delta$. We can repeat this argument up to $N(I)$ times, achieving
\begin{equation}\label{muN}
\mu_{N(I)} \leq u\circ G_{I,x_0}(w) \quad \mbox{for every}\quad w\in\overline{S}_{N(I),a\delta},
\end{equation}
where  $\mu_j = 2 \mu_{j-1} - \sup_{B_{x,\bar N\delta}}u$ for $j=1,\ldots,N(I)$. 
In particular, we have
\[
\mu_{N(I)}=2^{N(I)}u(x)-\Big(\sum_{j=0}^{N(I)-1}2^j\Big)M=2^{N(I)}u(x)-2^{N(I)}M+M
\]
with $M=\sup_{B_{x,\bar N\delta}}u$. In sum, we have proved that there exist $\delta_0>0$, $0<b<1$ and an integer $\bar N$ only depending on $K$ and $\cX$ such that for every $x\in K$, we can provide an integer $1\leq N_x\leq\bar N$, depending on $x$, 
such that for every $0<\delta<\delta_0$ we have $D_{x,\bar N\delta_0}\subset\Omega$ and 
\eqref{NxB} holds.
\end{proof}
\begin{Cor}\label{bddX} 
Under the assumptions of Lemma~\ref{lowbound}, we have
\begin{equation}\label{muN1}
\inf_{B_{x,b\delta}} u\geq\left\{\begin{array}{lll}
2\, u(x) - (2^{\bar N}-1)\sup_{B_{x,\bar N\delta}} u & \mbox{\rm if} & u(x)\geq0 \\
2^{\bar N}\, u(x) - (2^{\bar N}-1)\sup_{B_{x,\bar N\delta}} u & \mbox{\rm if} & u(x)<0\ \mbox{\rm and}\ \sup_{B_{x,\bar N\delta}}u\geq0\\
2^{\bar N}\, u(x) - \sup_{B_{x,\bar N\delta}} u & \mbox{\rm if} & \sup_{B_{x,\bar N\delta}}u<0
\end{array}\right.\,.
\end{equation}
\end{Cor}
%
%
%                 THEOREM ON LOCAL BOUNDEDNESS
%
%
The previous corollary immediately leads us to another consequence.
\begin{Cor}\label{locabbe}
Every $\cX$-convex function that is locally bounded from above on an open set is also locally bounded from below.
\end{Cor}
We use throughout the distance function $\dist_d(A,x)=\inf_{a\in A}d(a,x)$, with $A\subset\R^n$.
\begin{The}\label{LipReg}
Let $\cX=\{X_1,\ldots,X_m\}$ be a set of H\"ormander vector fields, let $\Omega\subset\R^n$ open and let $u : \Omega \rightarrow \R$ be a $\cX$-convex function that is locally bounded from above.
It follows that $u$ is locally Lipschitz continuous. More precisely, if $K\subset\Omega$ is compact and
$0<r<\mbox{\rm $\dist$}_d(K,\Omega^c)$, then for every $x,y\in K$ we have
\begin{equation}\label{LipEstd}
|u(x)-u(y)|\leq \frac{C}{r}\,d(x,y)\,\sup_{K_r}|u|\,,
\end{equation}
where $K_r=\{z\in\R^n:\mbox{\rm $\dist$}_d(K,z)\leq r\}\subset\Omega$ and $C>0$ only depends on $K$ and $\cX$.
\end{The}
\begin{proof}
First of all, from Corollary~\ref{locabbe} it follows that $u$ is locally bounded. 
Let us choose $0<D<\dist_d(K,\Omega^c)$ and consider the compact set
\[ K_D=\{z\in\R^n:\dist_d(K,z)\leq D\}\,,\]
that is clearly contained in $\Omega$. Choose any $\alpha>0$ such that 
$D+\alpha<\dist_d(K,\Omega^c)$. Therefore for every $x\in K_D$ and $X\in\cX$, we have
\[
\dist_d\big(K_D,\Phi^X(x,t)\big)\leq d(\Phi^X(x,t),x)\leq |t|\leq\alpha
\]
hence $\Phi^X(x,t)\in K_{D+\alpha}=\{z\in\R^n:\dist_d(K,z)\leq D+\alpha\}\subset\Omega$ for all $|t|\leq\alpha$. 
Hence $\Phi^X$ is defined on $K_D\times[-\alpha,\alpha]$ and it is contained in
the larger compact set $K_{D+\alpha}\subset\Omega$.
Let us fix $x,y \in K$ such that $\rho (x,y)< D$. Let $\ep>0$ be arbitrary chosen such that 
$\rho(x,y)+\ep<D$. Thus, by definition of $\rho$, there exists 
$\rho(x,y)<\bar t<\rho(x,y)+\ep$ and $\gamma\in\Gamma^c_{x,y}(\bar t)$ such that $t_0=0<t_1<\cdots<t_\nu=\bar t$ and
\begin{equation}\label{Phit_k}
\gamma(t)=\Phi^{X_{j_k}}\big(\gamma(t_{k-1}),t-t_{k-1}\big)
\end{equation}
for all $t\in[t_{k-1},t_k]$ and $k=1,\ldots,\nu$, where $1\leq j_1,\ldots,j_\nu\leq m$.
We have that 
\[ 
d(\gamma(t),x)\leq\rho(\gamma(t),x)\leq t\leq\bar t<D,
\]
therefore the whole curve $\gamma$ is contained in $K_D$ and any 
restriction $\gamma|_{[t_{k-1},t_k]}$ can be smoothly extended on $[t_{k-1}-\alpha,t_k+\alpha]$
preserving the same form \eqref{Phit_k}.
Since $u$ is locally bounded, we set 
\[ 
M=\sup_{w\in K_{D+\alpha}}|u(w)|<+\infty.
\]
As a result, the $\cX$-convexity of $u$ implies that the difference quotient
\[
\frac{|u(\gamma(t_k)) - u (\gamma(t_{k-1}))|}{|t_k - t_{k-1}|} 
\]
is not greater than the maximum between $\mbox{\scriptsize 
$|u\big(\Phi^{X_{j_k}}(\gamma(t_{k-1}),t_k+\alpha-t_{k-1})\big) - u (\gamma(t_k))|\, \alpha^{-1}$}$
and $|u\big(\Phi^{X_{j_k}}(\gamma(t_{k-1}),-\alpha)\big) - u (\gamma(t_{k-1}))|\alpha^{-1}$.
This yields proves that
\[
\frac{|u(\gamma(t_k)) - u (\gamma(t_{k-1}))|}{|t_k - t_{k-1}|} \leq \frac{2M}{\alpha}\,.
\]
It follows that
\[
|u(y) - u(x)| \leq \sum_{k=1}^\nu  |u(\gamma(t_k)) - u(\gamma(t_{k-1})| \leq \frac{2M}{T} 
\sum_{k=1}^\nu (t_k - t_{k-1})<\frac{2M}{\alpha}(\rho(x,y)+\ep)\,,
\]
with an arbitrary choice of $\ep>0$. In the case $\rho(x,y)\geq D$, we immediately have
$|u(x)-u(y)|\leq 2M\rho(x,y)/D$, that leads to the inequality
\[
|u(x)-u(y)|\leq \frac{2\rho(x,y)}{\min\{D,\alpha\}}\,\sup_{K_{D+\alpha}}|u|
\]
for every $x,y\in K$, where $D,\alpha>0$ satisfy $D+\alpha<\dist_d(K,\Omega^c)$.
Thus, we choose $r=2D=2\alpha<\dist_d(K,\Omega^c)$. By Proposition~\ref{equivrhod},
it follows that there exists a constant $C>0$, depending on $K$, such that
$4\rho(x,y)\leq C\,d(x,y)$ for all $x,y\in K$, hence concluding the proof.
\end{proof}
\begin{Rem}\rm
The Lipschitz estimate \eqref{LipEstd} restated with respect to the distance $\rho$
has only explicit constants. Precisely, under the assumptions of Theorem~\ref{LipReg} we have
\begin{equation*}
|u(x)-u(y)|\leq \frac{2\rho(x,y)}{\min\{\alpha_1,\alpha_2\}}\,\sup_{K_{\alpha_1+\alpha_2}}|u|\,.
\end{equation*}
\end{Rem}
%
%
%
%
%%%%%%%%%%%%%%%%%%%%%%%%%%%%%%%%%%%%%%%%%%%%%%%%%%%%%%%%%%%%%%%%%%%%%%%%%%%%%%%%%%%%%%%%%%%%%%%%%%%%%%%%%%%%
%
%
%
%
\section{$\cL$-weak subsolutions and upper estimates}\label{Sect:WeakSUI}
%
%
%
%
%%%%%%%%%%%%%%%%%%%%%%%%%%%%%%%%%%%%%%%%%%%%%%%%%%%%%%%%%%%%%%%%%%%%%%%%%%%%%%%%%%%%%%%%%%%%%%%%%%%%%%%%%%%%%
%
%
% 
% 
The point of this section is to show that locally bounded above $\cX$-convex functions are 
$\cL$-weak subsolutions of \eqref{subLap}, where $\cX=\{X_1,\ldots,X_m\}$ is a family of H\"ormander vector fields. 
This will enable us to apply the following well known result.
\begin{The}\label{upest}
Let $\Omega \subset \R^n$ be an open bounded set, and let $\cX$ be a family of smooth 
H\"ormander vector fields and let $p>0$. Thus, there exists $r_0>0$, depending on $\Omega$ and $\cX$,
and there exists $\kappa\geq1$, depending on $p$, $\Omega$ and $\cX$, such that whenever
$u \in W_\cX^{1,2}(\Omega)$ is a weak $\cL$-subsolution to \eqref{subLap}, we have
\begin{equation}\label{moser}
\mbox{\rm ess}\!\sup_{B_{x,\frac{r}{2}}}  u  \leq \kappa\; \left(\medint_{B_{x,r}} |u(y)|^p dy \right)^\frac{1}{p},
\end{equation}
for every $x \in \Omega$ such that $0<r\leq \min\{r_0,\mbox{\em dist}(\Omega^c,x)\}$.
\end{The}
The proof of this theorem is standard: it follows the celebrated Moser iteration technique for weak solutions to elliptic equations in divergence form \cite{Moser1960}, that applies to very general frameworks, including Carnot-Carath\'eodory spaces. There are several independent works in this area, so we limit ourselves to mention just a few of them, \cite{Lu1992}, \cite{HKM1993}, \cite{CDG93}. Further discussion of this topic can be found for instance in \cite{HajKos2000}.

In the proof of Theorem~\ref{issub}, we will use the following basic fact.
%
%
%               DIVERGENCE OF A VECTOR FIELD AND ITS FLOW
% 
%
\begin{Lem}\label{divX}
Let $X$ be a vector field on $\R^n$, let $z \in \R^n$ be such that $X(z) \neq 0$
and let $\pi$ be a hyperplane of $\R^n$ transversal to $X(z)$ and passing through $z$.
There exists an open neighbourhood $A$ of $z$ in $\pi$, $\tau>0$ and an open neighbourhood $U$ of
$z$ in $\R^n$ such that the restriction of the flow $\Phi^X$ to $A\times(-\tau,\tau)$ is a diffeomorphism onto $U$. 
Moreover, for every fixed system of coordinates $(\xi_1,\ldots,\xi_{n-1})$ on $\pi$, denoting by $\phi$ the previous 
restriction with respect to these coordinates and by $J_\phi$ its Jacobian, we get
\begin{equation}\label{divJ}
\div X(x)   = \frac{\partial_t J_\phi}{J_\phi} \circ \phi^{-1}(x)\quad\mbox{for all}\quad x\in U.
\end{equation}
\end{Lem}
\begin{Rem}\label{notzero}\rm
From the definition of commutator and the fact that the family $\cX$ satisfies the H\"ormander condition, 
it is clear that for each $z \in \R^n$, there exists $X\in\cX$ such that $X(z) \neq 0$.
\end{Rem}
\begin{proof}[Proof of Theorem~\ref{issub}]
As observed in Remark~\ref{notzero}, since $\cX$ is a family of H\"ormander vector fields, 
we must have some $j_1\in\{1,2,\ldots,m\}$ such that $X_{j_1}(x_0) \neq 0$. 
Thus, for each $i=1,\ldots,m$, we define $Y_i = X_i$ if $X_i(x_0) \neq 0$ and $Y_i =X_i+X_{j_1}$ otherwise,
so that all $Y_i$ do not vanish on $x_0$. In view of Lemma~\ref{divX}, for each $i=1,\ldots,m$ we can find 
an open bounded neighbourhood $U_i$ of $x_0$, that is compactly contained in $\Omega$,
an open bounded set $A_i\subset\R^{n-1}$, $\tau_i>0$  
and a diffeomorphism $\phi_i:S_i\ra U_i$, with $S_i=A_i\times(-\tau_i,\tau_i)$, 
$\phi_i$ is the restriction of the flow of $Y_i$ and then it satisfies \eqref{divJ}.
We can find $\delta_0>0$ such that $B_{x_0,\delta_0}$ is compactly contained in $U_i$ for all $i=1,\ldots,m$.
Let us choose any $\varphi \in C^\infty_c(B_{x_0,\delta_0})$ with $\varphi\geq0$. 
Our claim follows if we prove that  
\begin{equation}\label{toprove}
\sum_{i = 1}^m \int_{B_{x_0,\delta_0}}  Y_iu(x)\ Y^*_i \varphi(x)\, dx \geq 0.
\end{equation}
We will prove a stronger fact, namely, the validity of
\[
\int_{B_{x_0,\delta_0}}  Y_iu(x)\ Y^*_i \varphi(x)\, dx \geq 0 \quad\mbox{for all}\quad i =1,\ldots,m\,.
\]
By definition of $\cX$-convexity, we have that $u(\phi_i(\omega,\cdot))$ is convex on the interval where it is
defined for all $i=1,\ldots,m$. By Theorem~\ref{LipReg}, $u$ is locally Lipschitz continuous with respect to 
$d$. Iterating Lemma~\ref{XbarX}, no more than $m-1$ times, and observing that $\cX_1=\{Y_1,\ldots,Y_m\}$ is 
also a family of H\"ormander vector fields, its associated distance $d_1$ is equivalent to $d$, that is 
obtained from $\cX$. Theorem \ref{existenceDer} and Proposition~\ref{dirDer} imply that 
$u \in W_{\cX,loc}^{1,\infty}(\Omega)$ and the pointwise derivative 
\[
\partial_{Y_i} u(x)  = \frac{d}{dt}u (\Phi^{Y_i}(x,t))|_{t = 0} 
\]
exists for almost every $x\in\Omega$ and coincides with the distributional derivative $Y_iu$,
up to a negligible set. In particular, there exists $L>0$ such that $|Y_iu| \leq L$ almost everywhere 
in $U_i$, where $Y_iu$ is the distributional derivative of $u$ along $Y_i$. 
Since $\phi_i$ sends negligible sets into negligible sets, we have that
\begin{equation}\label{Ydera}
\frac{\partial}{\partial s } u(\phi_i(\omega,s))|_{s=t} = \der_{Y_i}u(\phi(\omega,t))=Y_iu(\phi_i(\omega,t))
\end{equation}
for almost every $(\omega,t) \in S_i$. There exist $0<t_i<\tau_i$ such that $\phi(A_i\times(-t_i,t_i))=U_i'$
still contains $B_{x_0,\delta_0}$, hence for $\ep>0$ sufficiently small, we can consider
\[
(u\circ\phi_i)_\ep (\omega,t) = \int_{-\tau_i}^{\tau_i} (u\circ \phi_i)(\omega,s))\, \nu_\ep(t-s) ds,
\]
for all $t\in(-t_i,t_i)$, where $\nu_\ep$ are one dimensional mollifiers.
Since $(u\circ\phi_i)(\omega,\cdot)$ is convex on $(-\tau_i,\tau_i)$ it is also locally Lipschitz,
with distributional derivative. It follows that 
\begin{equation*}
\frac{\partial}{\partial t} (u\circ \phi_i)_\ep(\omega,t) = 
(\der_{Y_i}u\circ \phi_i)_\ep(\omega,t)
\end{equation*}
for all $\omega\in A_i$ and $t\in(-t_i,t_i)$. Due to \eqref{Ydera}, applying Fubini's theorem it follows
that for almost every $\omega\in A_i$ the pointwise derivative $\der_{Y_i}u(\omega,t)$ equals 
the distributional derivative $Y_iu(\omega,t)$ for almost every $t\in(-\tau_i,\tau_i)$,
that is precisely represented almost everywhere.
As a consequence, we have
\begin{equation}\label{poY}
\frac{\partial}{\partial t} (u\circ \phi_i)_\ep(\omega,t) = (\der_{Y_i}u\circ \phi_i)_\ep(\omega,t)=
\big((Y_iu)\circ\phi_i\big)_\ep(\omega,t)
\end{equation}
for almost every $\omega\in A_i$ and every $t\in(-t_i,t_i)$.
Since $(u\circ\phi)_\ep(\omega,\cdot)$ is smooth and convex for all $\omega\in A_i$, we achieve
\[
\int_{S_i'} \frac{\partial^2}{\partial t^2} (u\circ \phi_i)_\ep(\omega,t)\ \varphi(\phi(\omega,t)) J_{\phi_i}(\omega,t)
\,d\omega dt\geq 0
\]
where $S_i'=A_i\times(-t_i,t_i)$.
%
%
%
%{\scriptsize \\ Since  $\supp\, \varphi\subset B_{x_0,\delta_0}\subset \phi_i(S_i')$, where $S_i'=A_i\times(-t_i,t_i)$, we have that $\ph\circ\phi_i$ is compactly supported on $S_i'$, hence $\ph\circ\phi_i(\omega,\cdot)$ is compactly supported on $(-t_i,t_i)$ and by Fubini's theorem one can integrate by parts,...}
%
%
Integrating by parts, it follows that the previous nonnegative integral equals the following one
\[
-\int_{S_i'} \frac{\partial}{\partial t} (u\circ\phi_i)_\ep(\omega,t)\frac{\partial}{\partial t}\left \lbrace \varphi(\phi_i(\omega,t)) J_{\phi_i} \right \rbrace d\omega dt    
\]
that can be written as follows 
\[
-\int_{S_i'} \bigg(\frac{\partial}{\partial t} (u\circ\phi_i)_\ep(\omega,t) \frac{\partial}{\partial t} \varphi(\phi_i(\omega,t)) J_{\phi_i} + \frac{\partial}{\partial t} (u\circ\phi_i)_\ep(\omega,t)\, (\varphi\circ\phi_i)(\omega,t)) \frac{\partial}{\partial t} J_{\phi_i}\bigg)  d\omega dt
\]
Clearly, we have $\frac{\partial}{\partial t} ( \varphi \circ \phi_i)(\omega,t)  = (Y_i\varphi)(\phi_i(\omega,t))$,
hence by Lemma~\ref{divX}, we obtain
\[
-\int_{S_i'}\frac{\partial}{\partial t} (u\circ\phi_i)_\ep(\omega,t) 
\Big((Y_i\varphi)(\phi(\omega,t)) + (\mbox{div}Y_i \circ \phi_i)(\omega,t)\,
(\varphi\circ\phi_i)(\omega,t))\Big)  J_{\phi_i} d\omega dt \geq 0.
\]
%
%{\scriptsize \\ Since $\frac{\partial}{\partial t} (u\circ\phi_i)_\ep=(Y_iu)\circ\phi_i$ a.e. in $A_i\times(-t_i,t_i)$, $|Y_iu|\leq L$ a.e. in $U_i$ and also $\frac{\partial}{\partial t} (u\circ\phi_i)_\ep(\omega,t)\ra \frac{\partial}{\partial t} (u\circ\phi_i)(\omega,t)$ for almost every $(\omega,t) \in S_i'$ as  $\ep \ra 0^{+}$ we can apply the dominated convergence theorem....}
%
We can then pass to the limit as $\ep\to 0^+$, taking into account that $Y_iu\in L^\infty(U_i)$ and 
that both \eqref{Ydera} and \eqref{poY} hold, getting
\[
-\int_{\phi_i^{-1}(U_i')} (Y_iu)\circ\phi_i\ \left \lbrace (Y_i\varphi\circ\phi_i + (\mbox{div}Y_i \circ \phi_i)\ \varphi\circ\phi_i\right \rbrace\  J_{\phi_i}\ d\omega dt\geq 0.
\]
By a change of variables towards the former coordinates, we obtain 
\[
\begin{aligned}
-\int_{U_i'} Y_i u(x) \left \lbrace (Y_i\varphi)(x) + \mbox{div}Y_i(x)\, \varphi(x)\right \rbrace dx  
= \int_{B_{x_0,\delta_0}} Y_i u(x)\ Y_i^* \varphi (x)\, dx \geq 0\,,
\end{aligned}
\]
that establishes our claim.
\end{proof}
As a consequence of both Theorem~\ref{upest} and Theorem~\ref{issub}, we get the following consequence.
\begin{Cor}\label{up}
Let $\Omega \subset \R^n$ be open and let $p>0$. If $x\in\Omega$, then there exist
$\sigma_x,\delta_x>0$ and $\kappa_x\geq1$, depending on $x$, $\Omega$, $p$ and $\cX$, such that $B_{x,\delta_x}\subset\Omega$, $\sigma_x\leq\delta_x/2$ and whenever $u:\Omega\ra\R$ is $\cX$-convex and locally bounded from above, for all $y\in B_{x,\delta_x/2}$ and $0<r\leq\sigma_x$, we have
\begin{equation}\label{upeq}
\sup_{B_{y,\frac{r}{2}}} u  \leq \kappa_x \left(\, \medint_{B_{y,r}} |u(z)|^p dz \right)^\frac{1}{p}.
\end{equation}
\end{Cor}
\begin{proof} Let $x\in\Omega$ and and consider the corrisponding $\delta_x>0$ given by Theorem~\ref{issub}, such that $B_{x,\delta_x}\subset\Omega$ and $u$ is a weak subsolution of \eqref{subLapY} where the vector fields $Y_j$ depend on $x$. In view of Theorem~\ref{upest} applied to the open bounded set $B_{x,\delta_x}$, we get some constants $\kappa_x\geq1$ and $r_x>0$, depending on $B_{x,\delta_x}$, $p$, and the vector fields $Y_j$, such that there holds 
\begin{equation}\label{moser1} \mbox{\rm ess}\sup_{B_{y,\frac{r}{2}}}  u  \leq \kappa_x\; \left(\medint_{B_{y,r}} |u(z)|^p dz \right)^\frac{1}{p}\,, 
\end{equation} 
for all $0<r\leq \min\{r_x,\dist(B_{x,\delta_x}^c,y)\}$. Since for all $y\in B_{x,\delta_x/2}$, we have 
\[
\dist(B_{x,\delta_x}^c,y)\geq\delta_x/2,
\]
setting $\sigma_x=\min\{r_x,\frac{\delta_x}{2}\}$, then \eqref{moser1} holds for all $0<r\leq \sigma_x$ and all $y\in B_{x,\delta_x/2}$. 
\end{proof}
\begin{Rem}\rm 
Notice that we do not need to use the essential supremum in \eqref{upeq},
since $\cX$-convex functions that are locally bounded from above are locally
Lipschitz continuous, due to Theorem~\ref{LipReg}.
\end{Rem}
As a consequence of Corollary~\ref{up}, we can easily establish the following result.
\begin{The}\label{upK}
Let $\Omega \subset \R^n$ be open, let $p>0$ and let $K\subset\Omega$ be compact. 
Then there exists $\sigma>0$ and $\kappa\geq1$, depending on $K$, $\Omega,$ $\cX$ and $p$, such that 
for every $\cX$-convex function $u:\Omega\ra\R$ that is locally bounded from above and for 
every $x\in K$, we have $B_{x,\sigma}\subset\Omega$ and there holds
\begin{equation}\label{upeqK}
\sup_{B_{x,\frac{r}{2}}} u  \leq \kappa \left(\, \medint_{B_{x,r}} |u(z)|^p dy \right)^\frac{1}{p}
\quad \mbox{for all}\quad 0<r\leq\sigma.
\end{equation}
\end{The}
%
%
%
%
%{\scriptsize \begin{proof} We can cover $K$ by a finite number $N$ of metric balls $B_{x_1,\delta_{x_1}/2},\ldots, B_{x_N,\delta_{x_N}/2}$, where $\delta_{x_j}$ are given by Corollary~\ref{up}, along with the constants $\kappa_{x_i}$ and $\sigma_{x_i}$. Therefore for any $y\in K$, we have $y\in B_{x_i,\delta_{x_i}/2}$ for some $i\in\{1,2,\ldots,N\}$. Again, by Corollary~\ref{up} , we have \begin{equation}\label{upeqi} \sup_{B_{y,\frac{r}{2}}} u  \leq \kappa_{x_i} \left(\, \medint_{B_{y,r}} |u(z)|^p dy \right)^\frac{1}{p} \end{equation} for all $0<r\leq\sigma_{x_i}$. Setting $\sigma=\min\{\sigma_{x_1},\ldots,\sigma_{x_N}\}$ and $\kappa=\max\{\kappa_{x_1},\ldots,\kappa_{x_N}\}$, our claim follows. \end{proof} }
%
%
%
%
%%%%%%%%%%%%%%%%%%%%%%%%%%%%%%%%%%%%%%%%%%%%%%%%%%%%%%%%%%%%%%%%%%%%%%%%%%%%%%%%%%%%%%%%%%%%%%%%%%%%%%%%%%%%
%
%
%
%
\section{Regularity estimates for $\cX$-convex functions}\label{Sect:IntEst}
%
%
%
%
%%%%%%%%%%%%%%%%%%%%%%%%%%%%%%%%%%%%%%%%%%%%%%%%%%%%%%%%%%%%%%%%%%%%%%%%%%%%%%%%%%%%%%%%%%%%%%%%%%%%%%%%%%%%%
%
%
% 
% 
In this section we combine the upper and lower estimates for $\cX$-convex functions, that give the proof
of Theorem~\ref{Thm:IE}.
\begin{The}\label{|estim|}
Let $\Omega\subset\R^n$ be open, let $K\subset\Omega$ be compact and let 
$u:\Omega\ra\R$ be a $\cX$-convex function that is locally bounded from above. 
Then there exists $C_0>0$,  $b_0>0$ and $N_0>1$, depending on $K$, such that
for every $x\in K$ there holds
\[
\sup_{B_{x,r}}|u|\leq C_0\,\medint_{B_{x,N_0r}}|u(z)|\,dz
\]
whenever $0<r<b_0$ and $K_0=\{z\in\R^n:\mbox{\rm dist}(K,z)\leq N_0\,	b_0\}\subset\Omega$.
\end{The}
\begin{proof}
By Lemma~\ref{lowbound}, we have $\delta_0>0$, $0<b<1$ and a positive integer $\bar N$ such that
for every $y\in K$, we have $D_{y,\bar N\delta_0}\subset\Omega$ and there exists with 
$1\leq N_y\leq\bar N$ such that
\begin{equation}\label{Nbarx}
2^{N_y}\, u(y) - (2^{N_y}-1)\sup_{B_{y,\bar N\delta}} u \leq \inf_{B_{y,b\delta}} u
\end{equation}
for all $0<\delta<\delta_0$. Let us consider $x\in K$ and any $0<\delta'<b\delta_0/4$,
observing that there exists $x'\in B_{x,\delta'}$ such that
\[
u(x')\geq-\,\medint_{B_{x,\delta'}}|u(z)|\,dz\,.
\]
We clearly have $\inf_{B_{x,\delta'}} u\geq \inf_{B_{x',2\delta'}}u$, hence for some
$1\leq N_{x'}\leq\bar N$, we can apply the estimate \eqref{Nbarx} at $x'$, getting
\[
\inf_{B_{x,\delta'}} u\geq 2^{N_{x'}}u(x')-(2^{N_{x'}}-1)\sup_{B_{x',\bar N\frac{2\delta'}{b}}}u.
\]
From the previous inequalities, it follows that
\[
\inf_{B_{x,\delta'}} u\geq -2^{\bar N}\medint_{B_{x,\delta'}}|u(z)|\,dz-
(2^{N_{x'}}-1)\sup_{B_{x,\bar N\frac{4\delta'}{b}}}u.
\]
Theorem~\ref{upK} provides $\sigma>0$ and $\kappa\geq1$ such that, up to choose $\delta_0>0$
possibly smaller, such that $\bar N\delta_0<\sigma/2$, hence $\bar N\frac{8\delta'}{b}<\sigma$
and it follows that
\[
\inf_{B_{x,\delta'}} u\geq -2^{\bar N}\medint_{B_{x,\delta'}}|u(z)|\,dz-
(2^{\bar N}-1)\,\kappa\, \medint_{B_{x,\bar N\frac{8\delta'}{b}}}|u(z)|\,dz.
\]
As a consequence of Corollary~\ref{doubling}, we have $Q_0$
%{\scriptsize $=\log_2\bar C_0$}
$>0$ and $r_0>0$
%{\scriptsize \\ (such that for all $y\in K$ we have \[ |B_{y,R}|\leq 2^Q (R/r)^Q\,|B_{y,r}| \] whenever $0<r<R<r_0$ this gives)  }
such that
\[
|B_{x,\bar N\frac{8\delta'}{b}}|\leq 2^{Q_0}\,\Big(\bar N\frac{8}{b}\Big)^{Q_0}\,|B_{x,\delta'}|,
\]
up to making $\delta_0$ further smaller, namely, satisfying $2\bar N\delta_0<r_0$.
It follows that 
\[
\inf_{B_{x,\delta'}} u\geq -2^{\bar N} \bigg[\kappa+(16)^{Q_0}\Big(\frac{\bar N}{b}\Big)^{Q_0}\bigg] 
\;\medint_{B_{x,\bar N\frac{8\delta'}{b}}}|u(z)|\,dz
\]
and also
\[
\sup_{B_{x,\delta'}}u\leq \kappa\, 2^{Q_0}\,\Big(\bar N\frac{4}{b}\Big)^{Q_0}\,
\medint_{B_{x,\bar N\frac{8\delta'}{b}}}|u(z)|\,dz\,,
\]
that yield a constant $C_0>0$ depending on $K$, such that
\[
\sup_{B_{x,r}}|u|\leq C_0\,\medint_{B_{x,N_0r}}|u(z)|\,dz
\]
for every $0<r<b_0$ and every $x\in K$, with $b_0=b\delta_0/4$ and $N_0=\bar N\frac{8}{b}>1$.
By the previous requirements on $\delta_0$, being $N_0b_0=2\delta_0\bar N$,
we also have 
\[
K_0=\{z\in\R^n:\dist(K,z)\leq N_0b_0\}\subset\Omega,
\]
reaching the conclusion of the proof.
\end{proof}
\begin{The}\label{Thm:lambda}
Let $\Omega\subset\R^n$ be open, let $K\subset\Omega$ be compact and let $\lambda>1$. 
Then there exist $\bar C>0$ and $\bar Q>0$, depending on $K$ and there there exists
$\bar r>0$, depending on both $K$ and $\lambda$, such that for every 
$x\in K$ and every $0<r<\bar r$, each $\cX$-convex function $u:\Omega\ra\R$, that is locally bounded from above satisfies the following estimate
\begin{equation}\label{estlamb}
\sup_{B_{x,r}} |u|\leq \bar C\,\bigg(\frac{\lambda+1}{\lambda-1}\bigg)^{\bar Q}\ \medint_{B_{x,\lambda r}}|u(z)|\,dz\,.
\end{equation}
\end{The}
\begin{proof}
We fix any $\beta>0$ such that $K_1=\{z\in\R^n:\mbox{\rm dist}(K,z)\leq \beta\}\subset\Omega$
and apply Theorem~\ref{|estim|} to $K_1$, getting the corresponding positive constants $C_1,b_1$
and $N_1>1$. We have in particular 
\[
\{z\in\R^n:\mbox{\rm dist}(K_1,z)\leq N_1\,	b_1\}\subset\Omega.
\]
Taking $0<r<\beta/\lambda$, we have $B_{x,\lambda r}\subset K_1$ for all $x\in K$
and fixing $a=(\lambda-1)/N_1$, it follows that for $0<r<r_1$ and $r_1=\min\{b_1/a,\beta/\lambda\}$,
the following inequality
\[
\sup_{B_{y,ar}}|u|\leq C_1\,\medint_{B_{y,N_1ar}}|u(z)|\,dz
\]
holds for all $y\in K_1$. Now, let us fix $x\in K$. Thus, whenever $0<r<r_1$ we can cover 
the compact set $D_{x,r}$ with a finite number of balls $B_{x_j,ar}$ centered at points of $D_{x,r}$, 
hence there exists $x_{j_0}\in D_{x,r}$ such that
\[
\sup_{B_{x,r}} |u|\leq \sup_{B_{x_{j_0}},ar} |u|\,.
\] 
Since $x_{j_0}\in K_1$ and $ar<b_1$, Theorem~\ref{|estim|} implies that
\[
\sup_{B_{x_{j_0}},ar} |u|\leq C_1\,\medint_{B_{x_{j_0},N_1ar}}|u(z)|\,dz=
C_1\,\medint_{B_{x_{j_0},(\lambda-1) r}}|u(z)|\,dz\,.
\]
As a result, we have proved that 
\[
\sup_{B_{x,r}} |u|\leq C_1\,\frac{|B_{x,\lambda r}|}{|B_{x_{j_0},(\lambda-1) r}|}\
\medint_{B_{x,\lambda r}}|u(z)|\,dz
\leq C_1\,\frac{|B_{x_{j_0},(\lambda+1)r}|}{|B_{x_{j_0},(\lambda-1) r}|}\
\medint_{B_{x,\lambda r}}|u(z)|\,dz
\]
for all $0<r<r_1$, where $r_1$ also depends on $\lambda$. 
Finally, we apply Corollary~\ref{doubling} to $K_0$, getting $r_2>0$ and $\bar Q>0$ such that 
for all $0<r<\min\{r_1,r_2/\lambda+1\}$ our claim \eqref{estlamb} holds with $\bar C=C_1\,2^{\bar Q}$.
\end{proof}

\end{document}